\documentclass{elsarticle}

\usepackage[utf8]{inputenc}
\usepackage{physics}
\usepackage{amsmath}
\usepackage{amssymb}
\usepackage{graphicx}
\usepackage{subcaption}
\usepackage{bbm}
\usepackage{xcolor}
\usepackage{colortbl}
\usepackage{siunitx}
\usepackage{hyperref}
\usepackage{etoolbox}
\usepackage{appendix}
\usepackage{cleveref}

\AtBeginEnvironment{appendices}{\crefalias{section}{appendix}}

\journal{Journal of Computational and Applied Mathematics}
\date{June 7, 2021}

\newcommand{\qi}{\mathbf{i}}
\newcommand{\qj}{\mathbf{j}}
\newcommand{\qk}{\mathbf{k}}
\newcommand{\qqq}{\mathbf{q}}
\newcommand{\qI}[1]{\mathbf{I}_{#1}}
\newcommand{\qJ}[1]{\mathbf{J}_{#1}}
\newcommand{\qK}[1]{\mathbf{K}_{#1}}
\newcommand{\qQ}[1]{\mathbf{Q}_{#1}}

\newcommand{\spacealg}{\mathcal{G}(\mathbb{R}_{3})}

\DeclareMathOperator{\atantwo}{atan2}
\DeclareMathOperator\arctanh{arctanh}
\DeclareMathOperator\Ln{Ln}
\DeclareMathOperator\Arg{Arg}

\newtheorem{thm}{Theorem}
\newtheorem{lemma}[thm]{Lemma}
\newproof{proof}{Proof}
\newproof{pot}{Proof of Theorem \ref{thm2}}

\begin{document}

\begin{frontmatter}

\title{Quaternionic Step Derivative: Machine Precision Differentiation of Holomorphic Functions using Complex Quaternions}

\author[1]{Martin Roelfs\corref{cor1}
\fnref{fn1}}
\ead{martin.roelfs@kuleuven.be}
\author[1,2]{David Dudal\fnref{fn1}}
\ead{david.dudal@kuleuven.be}
\author[3]{Daan Huybrechs\fnref{fn2}}
\ead{daan.huybrechs@kuleuven.be}

\cortext[cor1]{Corresponding author}
\fntext[fn1]{The research of M.R.~and D.D.~is supported by KU Leuven IF project C14/16/067.}
\fntext[fn2]{The research of D.H.~is supported by KU Leuven IF project C14/15/055.}
\address[1]{KU Leuven Campus Kortrijk--Kulak, Department of Physics, Etienne Sabbelaan 53 bus 7657, 8500 Kortrijk, Belgium}
\address[2]{Ghent University, Department of Physics and Astronomy, Krijgslaan 281-S9, 9000 Gent, Belgium}
\address[3]{KU Leuven, Department of Computer Science, Celestijnenlaan 200A, BE-3001 Leuven, Belgium}

\begin{abstract}
The known Complex Step Derivative (CSD) method allows easy and accurate differentiation up to machine precision of real analytic functions by evaluating them a small imaginary step next to the real number line. 
The current paper proposes that derivatives of holomorphic functions can be calculated in a similar fashion by taking a small step in a quaternionic direction instead.
It is demonstrated that in so doing the CSD properties of high accuracy and convergence are carried over to derivatives of holomorphic functions. 
To demonstrate the ease of implementation, numerical experiments were performed using complex quaternions, the geometric algebra of space, and a $2 \times 2$ matrix representation thereof.
\end{abstract}

\begin{keyword}
  Automatic Differentiation \sep Complex Analysis \sep Complex Step Derivative \sep Complex Quaternions \sep Geometric Algebra \sep Clifford Algebra
\end{keyword}

\end{frontmatter}


\section{Introduction}

The Complex Step Derivative (CSD) of a real analytic function $f(x)$ works by evaluating the function not on the real number line, but by taking a small imaginary step $ih$ beside it instead~\cite{squire}. From the Taylor series we see that
    \begin{equation}
        f(x+ih) = f(x) + i h f'(x) - \frac{h^2}{2} f''(x) + \order{h^3}.
    \end{equation}
Therefore, if $h$ is chosen sufficiently small the imaginary part gives the first derivative up to corrections of $\order{h^2}$:
    \begin{equation}
         f'(x) = \frac{1}{h}\Im\bqty{f(x+ih)} + \order{h^2}.
         \label{eq:csd_derivative}
    \end{equation}
If the function can be accurately evaluated at the complex point $x+ih$, this method is actually convergent and therefore $h$ can be made arbitrarily small, leading to derivatives accurate to machine precision. This is not the case for finite difference based techniques, since eventually $f(x)$ and $f(x+h)$ become indistinguishable and therefore $f(x+h) - f(x)$ equals zero in floating point arithmetic.

Although very accurate and easily implementable, CSD does not work for complex functions $f: \mathbb{C} \to \mathbb{C}$, since the imaginary part is no longer directly identifiable with the derivative.
However, it is possible to recover a similar rule by taking a small step in a quaternionic direction instead:
    \begin{equation}\label{eq:quaternion_step}
        f(z+\qj h) = f(z) + h \qj f'(z) - \frac{h^2}{2} f''(z) + \order{h^3},
    \end{equation}
where $\qj$ is a unit quaternion which, like the imaginary unit, satisfies $\qj^2=-1$. This enables instant derivation of holomorphic functions and, by simple extension, of piecewise holomorphic functions (piecewise functions whose pieces are complex differentiable, i.e., analytic). Just like CSD this technique is convergent and in many cases gives the first derivative up to machine precision with only a single function evaluation. This method will be referred to as Quaternionic Step Derivation (QSD).

Compared to existing generalizations of CSD, which we briefly review further on, QSD can be implemented using standard libraries and available software. If software support for quaternions is not available, we will show that one can switch to matrix representations: derivatives can be evaluated using a single function evaluation with a matrix argument. Support for polynomials and special functions with matrix arguments is included in several existing scientific libraries, such as \verb|expm|, \verb|sinm| and \verb|cosm| in \verb|SciPy|, or \verb|expm| in \verb|Matlab|.

This paper is organized as follows. Firstly, \S\ref{sec:qsd} introduces the quaternions using the geometric algebra of 3 dimensional space, discusses their matrix representations, and the formal conditions under which QSD is valid. Secondly, \S\ref{sec:results} discusses the application of QSD to several trial functions. Lastly, we summarize the most important conclusions in \S\ref{sec:conclusion}.

\section{Quaternionic Step Derivation} \label{sec:qsd}
The goal is to find the derivative of an analytic function $f(z) \in \mathbb{C}$ with $z \in \mathbb{C}$. In order to do this we would like to introduce another imaginary unit $\qj$ such that
    \begin{equation}
        f(z+h\qj) = f(z) + h\qj f'(z)  - \frac{h^2}{2} f''(z) + \order{h^3}.
        \label{eq:qsd_def}
    \end{equation}
The obvious thing to try is to introduce the quaternions $\qi$, $\qj$ and $\qk$ which all satisfy $\qi^2 = \qj^2 = \qk^2 = -1$ and to build the complex numbers as $z = a + b \qi$ and to use e.g. $\qj$ as the new variable. However, this does not work because $\qi$~and~$\qj$ anti-commute: $\qi\qj + \qj\qi = 0$. 
As a consequence the square of any normalized quaternion $\qqq = c_\qi \qi + c_\qj \qj + c_\qk \qk$ with $c_\qi^2 + c_\qj^2 + c_\qk^2 = 1$ is $-1$:
    \[ \qqq = \pqty{c_\qi \qi + c_\qj \qj + c_\qk \qk}^2 = -1. \]
Ergo, the quaternion $\qqq$ as a whole behaves as a unit imaginary, and $\qi$~and~$\qj$ cannot be considered as separate imaginary units.
Therefore we need two commuting imaginary units.

A straightforward way to do this is to introduce multicomplex numbers using the recursive definition
\begin{equation}
    \mathbb{C}^n = \Bqty{z_1 + z_2 i_n \vert z_1, z_2 \in \mathbb{C}^{n-1}},
\end{equation}
where $i_n^2=-1$ and $\mathbb{C}^0 := \mathbb{R}$. Lantoine et al. have demonstrated the utility of this approach \cite{Lantoine} for taking higher order derivatives of real analytic functions, and it is clear that defining bicomplex numbers
\begin{equation}
    \mathbb{C}^2 = \Bqty{(a_1 + i_1 b_1) + i_2 (a_2 + i_1 b_2) \, \vert \, a_1, a_2, b_1, b_2 \in \mathbb{R}, i_1^2 = i_2^2 = -1}
\end{equation}
allows the computation of the first derivative of a complex function using \eqref{eq:qsd_def}.

However, an alternative solution becomes apparent by considering the geometric algebra $\spacealg$, whose bivectors are the quaternions, and whose pseudoscalar is a commuting imaginary unit \cite{GA4CS,STA}. The connection between $\spacealg$ and quaternions is well-known, but we include a derivation here from basic principles. This allows us to shed light on the choices one has for a quaternionic direction, which we will explore further afterwards.

\subsection{Commuting imaginary units based on geometric algebra}

The geometric algebra over a vector space distinguishes itself by the definition of the geometric product, which combines the dot and exterior product between vectors, to multiply vectors into multivectors. 
Typically, and for the remainder of this article, the geometric product is written using juxtaposition.

In order to construct $\spacealg$, we start with the orthogonal basis vectors $e_i$ with $i=1,2,3$, and impose that $e_i^2 = 1$. 
But since there is nothing special about these particular basis vectors, any linear combination of them must also square to a scalar. This is called the \emph{contraction axiom}: if $x = ae_1 + be_2 + ce_3$, then $x^2 \in \mathbb{R}$, where $a,b,c \in \mathbb{R}$. Explicit calculation gives
    \begin{equation}
        x^2 = \pqty{a^2 + b^2 + c^2} + ab \pqty{e_{12} +e_{21}} + ac \pqty{e_{13} +e_{31}} + bc \pqty{e_{23} + e_{32}}.
    \end{equation}
Thus, $e_{i}e_{j} = - e_{j}e_{i}$ in order for $x^2$ to be a real number. Hence, orthogonal vectors anti-commute. The product of two orthogonal basis vectors does not reduce to a scalar but is a \emph{bivector}:
    \begin{equation}
        e_{ij} := e_i e_j.
    \end{equation}
Using the anticommutivity of orthogonal vectors we compute the square of the bivectors and find
	\begin{equation}
		e_{ij}^2 = e_{i} e_{j} e_{i} e_{j}= - e_{i}^2 e_{j}^2 = -1.
	\end{equation}
The bivectors are therefore of imaginary type and can be associated with the quaternions: 
    \begin{equation}
        \qi := e_{32}, \quad \qj := e_{13}, \quad \qk := e_{21}.
    \end{equation}
But it is also possible to multiply all three basis vectors together, which gives the \emph{pseudoscalar} $e_{123}$:
    \begin{equation}
        i := e_1 e_2 e_3 = e_{123}.
    \end{equation}
This element not only squares to $-1$, but it also commutes with all elements of the algebra, which is straightforward to demonstrate using the anti-commutivity of orthogonal vectors. This makes it a perfect candidate for the regular complex unit $i := e_{123}$.
In terms of $i=e_{123}$ the quaternions can also be written as
    \begin{equation}
        \qi = -i e_1, \quad \qj = -i e_2, \quad \qk = -i e_3.
        \label{eq:quaternions}
    \end{equation}
The Cayley table of $\spacealg$ is shown in \cref{tab:cayley}.
    \begin{table}
        {\footnotesize \caption{Cayley table of $\spacealg$. This algebra has four elements of imaginary type. The header shows the identification with $\qi, \qj, \qk$ and $i$. Table generated using http://bivector.net.\label{tab:cayley}}}
        \centering
        \begin{tabular}{c|c|ccccccc|}
            \multicolumn{2}{c}{}& $i \qk$ & $i \qj$ & $i \qi$ & $-\qi$ & $\qj$ & $-\qk$ & \multicolumn{1}{c}{$i$} \\
            \cline{2-9}\cline{2-9}
            & \cellcolor[HTML]{88FF88}$1$ & \cellcolor[HTML]{CCCCFF}$e_{1}$ & \cellcolor[HTML]{CCCCFF}$e_{2}$ & \cellcolor[HTML]{CCCCFF}$e_{3}$ & \cellcolor[HTML]{FFCCCC}$e_{12}$ & \cellcolor[HTML]{FFCCCC}$e_{13}$ & \cellcolor[HTML]{FFCCCC}$e_{23}$ & \cellcolor[HTML]{FFCCFF}$e_{123}$\\
            \cline{2-9}
$i \qk$ & \cellcolor[HTML]{CCCCFF}$e_{1}$ & \cellcolor[HTML]{88FF88}$1$ & \cellcolor[HTML]{FFCCCC}$e_{12}$ & \cellcolor[HTML]{FFCCCC}$e_{13}$ & \cellcolor[HTML]{CCCCFF}$e_{2}$ & \cellcolor[HTML]{CCCCFF}$e_{3}$ & \cellcolor[HTML]{FFCCFF}$e_{123}$ & \cellcolor[HTML]{FFCCCC}$e_{23}$\\
$i \qj$ & \cellcolor[HTML]{CCCCFF}$e_{2}$ & \cellcolor[HTML]{FFAAAA}$-e_{12}$ & \cellcolor[HTML]{88FF88}$1$ & \cellcolor[HTML]{FFCCCC}$e_{23}$ & \cellcolor[HTML]{AAAAFF}$-e_{1}$ & \cellcolor[HTML]{FFAAFF}$-e_{123}$ & \cellcolor[HTML]{CCCCFF}$e_{3}$ & \cellcolor[HTML]{FFAAAA}$-e_{13}$\\
$i \qi$ & \cellcolor[HTML]{CCCCFF}$e_{3}$ & \cellcolor[HTML]{FFAAAA}$-e_{13}$ & \cellcolor[HTML]{FFAAAA}$-e_{23}$ & \cellcolor[HTML]{88FF88}$1$ & \cellcolor[HTML]{FFCCFF}$e_{123}$ & \cellcolor[HTML]{AAAAFF}$-e_{1}$ & \cellcolor[HTML]{AAAAFF}$-e_{2}$ & \cellcolor[HTML]{FFCCCC}$e_{12}$\\
$- \qi$ & \cellcolor[HTML]{FFCCCC}$e_{12}$ & \cellcolor[HTML]{AAAAFF}$-e_{2}$ & \cellcolor[HTML]{CCCCFF}$e_{1}$ & \cellcolor[HTML]{FFCCFF}$e_{123}$ & \cellcolor[HTML]{FF8888}$-1$ & \cellcolor[HTML]{FFAAAA}$-e_{23}$ & \cellcolor[HTML]{FFCCCC}$e_{13}$ & \cellcolor[HTML]{AAAAFF}$-e_{3}$\\
$\qj$ & \cellcolor[HTML]{FFCCCC}$e_{13}$ & \cellcolor[HTML]{AAAAFF}$-e_{3}$ & \cellcolor[HTML]{FFAAFF}$-e_{123}$ & \cellcolor[HTML]{CCCCFF}$e_{1}$ & \cellcolor[HTML]{FFCCCC}$e_{23}$ & \cellcolor[HTML]{FF8888}$-1$ & \cellcolor[HTML]{FFAAAA}$-e_{12}$ & \cellcolor[HTML]{CCCCFF}$e_{2}$\\
$- \qk$ & \cellcolor[HTML]{FFCCCC}$e_{23}$ & \cellcolor[HTML]{FFCCFF}$e_{123}$ & \cellcolor[HTML]{AAAAFF}$-e_{3}$ & \cellcolor[HTML]{CCCCFF}$e_{2}$ & \cellcolor[HTML]{FFAAAA}$-e_{13}$ & \cellcolor[HTML]{FFCCCC}$e_{12}$ & \cellcolor[HTML]{FF8888}$-1$ & \cellcolor[HTML]{AAAAFF}$-e_{1}$\\
$i $ & \cellcolor[HTML]{FFCCFF}$e_{123}$ & \cellcolor[HTML]{FFCCCC}$e_{23}$ & \cellcolor[HTML]{FFAAAA}$-e_{13}$ & \cellcolor[HTML]{FFCCCC}$e_{12}$ & \cellcolor[HTML]{AAAAFF}$-e_{3}$ & \cellcolor[HTML]{CCCCFF}$e_{2}$ & \cellcolor[HTML]{AAAAFF}$-e_{1}$ & \cellcolor[HTML]{FF8888}$-1$\\
        \cline{2-9}
        \end{tabular}
    \end{table}
    
There are now several choices for a quaternionic generalization of complex numbers that we could use. For readibility, here we choose $\qj$ as the quaternionic component. Thus, a bicomplex number $w$ can be written with a quaternionic component as
    \begin{equation}
        w = (a + b i) + (c + d i) \qj.
    \end{equation}
The number of basis vectors in an element defines its grade, e.g. $\qj = e_{13}$ is of grade~2 and $i = e_{123}$ is of grade~3.
As a generalization of the usual $\Re$ and $\Im$ operators of complex analysis it is prudent to introduce a grade selection operator to distinguish the various grades.
The grade selection operator is denoted $\expval{w}_r$, and selects the grade $r$ part of $w$. Using grade selection we can select the complex part of $w$ by defining
    \[ \expval{w}_\mathbb{C} := \expval{w}_0 + \expval{w}_3 = a + bi. \]
With this notation the QSD analog of \eqref{eq:csd_derivative} is
    \begin{equation}
        f'(z + h \qj) = \frac{1}{h} \expval{\frac{f(z + h\qj)}{\qj}}_\mathbb{C} + \order{h^2}.
    \end{equation}

Because transcendental functions like exponentials and logarithms are well defined for complex quaternions, QSD can be used to evaluate their derivatives as well. 
However, depending on the library or representation used, the promised machine level precision might not always be available out of the box. 
For this reason \cref{app:elementary_functions} details closed form expressions for these functions.

\subsection{QSD is agnostic to the quaternionic direction} \label{sec:agnostic}
In the previous section $\qj$ was chosen as the quaternionic direction to aid readability. However, in principle there is no preference for a particular quaternionic direction. To demonstrate this, let $\qqq$ be a linear combination of the unit quaternions:
    \begin{equation}
        \qqq = c_\qi \qi + c_\qj \qj + c_\qk \qk,
    \end{equation}
normalized such that $\qqq^2 = -1$ and therefore $c_\qi^2+c_\qj^2+c_\qk^2 = 1$.
Since $i = e_{123}$ commutes with all quaternions, it also commutes with $\qqq$, and thus $\qqq$ can be used to perform QSD:
    \begin{align}\label{eq:taylor_quaternion}
        f(z + h \qqq) &= f(z) + h \qqq f'(z) - \frac{h^2}{2} f''(z) + \order{h^3}.
    \end{align}
Since $\qqq^2 = -1$, $f'(z)$ is obtained as
    \begin{equation*}
        f'(z) = \frac{1}{h} \expval{f(z + h \qqq) \qqq^{-1}}_\mathbb{C} + \order{h^2}.
    \end{equation*}
Expressing $c_\qi, c_\qj$ and $c_\qk$ in spherical coordinates,
    \begin{align}
        c_\qi &= \sin \theta \cos \phi, \qquad c_\qj = \sin \theta \sin \phi, \qquad c_\qk = \cos \theta,
        \label{eq:quat_coeff}
    \end{align}
with $\theta \in [0, \pi], \phi \in [0, 2 \pi )$, the claim of independence of quaternionic directions can be investigated as a function of $(\theta, \phi)$. The results are presented further on in \S\ref{sec:results_agnostic}, and have interesting consequences for the $2 \times 2$ matrix representation.

\subsection{Matrix representations of \texorpdfstring{$\spacealg$}{3D space}}\label{sec:matrix_repr}

In order to build a matrix representation of $\spacealg$ we use that the Pauli matrices~$\sigma_i$ form a matrix representation of the basis vectors of 3D space \cite{STA}, and therefore $\sigma_i \Leftrightarrow e_i$, where $\Leftrightarrow$ denotes isomorphism. It follows that the pseudoscalar $e_{123}$ is identified with 
\[ e_{123} \Leftrightarrow \sigma_1 \sigma_2 \sigma_3 = i \mathbbm{1}_2, \]
where $i$ is the ordinary complex unit and $\mathbbm{1}_2$ the $2 \times 2$ identity matrix. Thus,
    \begin{align}
        \qi \Leftrightarrow \qI{2} := -i \sigma_1 &= \mqty(0 & -i \\ -i & 0), \quad \notag \\
        \qj \Leftrightarrow \qJ{2} := -i \sigma_2 &= \mqty(0 & -1 \\ 1 & 0), \quad \notag \\
        \qk \Leftrightarrow \qK{2} := -i \sigma_3 &=\mqty(-i & 0 \\ 0 & i). \label{eq:defining_rep_su2}
    \end{align}
This representation is compact, but it does suffer from a drawback. In the geometric algebra $\spacealg$ all basis elements are orthogonal under the inner product:
    \[x \cdot y = \begin{cases}0 & x \neq y \\ \pm 1 & x = y \end{cases},\]
as is evident from the Cayley \cref{tab:cayley}. The matrix equivalent of this inner product is
    \begin{equation}
        x \cdot y \to \frac{1}{n} \trace{X_n Y_n},
        \label{eq:inner_product_mat}
    \end{equation}
where $X_n$ is an $n \times n$ matrix representation of $x$. When considering $A = X_2 Y_2$ with $X_2 = \qK{2}$ and $Y_2 = i \mathbbm{1}_2$, we find
    \[ A_{11} + A_{22} = \frac{1}{2} \Tr\pqty{i \qK{2}} = \frac{1}{2} \Tr\mqty(1 & 0 \\ 0 & -1), \]
which is zero analytically, but numerically the exact cancellation of $A_{11}$ and $A_{22}$ is not guaranteed. As will be demonstrated in \S\ref{sec:results_agnostic} and \S\ref{sec:relation_finite_diff}, in practice QSD implementations involving $\qK{2}$ lead to catastrophic cancellations. When $\qI{2}$ or $\qJ{2}$ is used this is not an issue as the zero trace does not depend on the exact cancellation of the elements. 

In order to restore the orthogonality of the basis elements of $\spacealg$ a $4 \times 4$ real representation of the quaternions can be used instead:
    \begin{align}
        \qi \Leftrightarrow \qI{4} &:= \mqty(0 & 1 & 0 & 0 \\ -1 & 0 & 0 & 0 \\ 0 & 0 & 0 & 1 \\ 0 & 0 & -1 & 0), \quad \notag \\
        \qj \Leftrightarrow \qJ{4} &:= \mqty(0 & 0 & -1 & 0 \\ 0 & 0 & 0 & 1 \\ 1 & 0 & 0 & 0 \\ 0 & -1 & 0 & 0), \quad \notag \\
        \qk \Leftrightarrow \qK{4} &:= \mqty(0 & 0 & 0 & 1 \\ 0 & 0 & 1 & 0 \\ 0 & -1 & 0 & 0 \\ -1 & 0 & 0 & 0). \label{eq:matrix_rep_4d}
    \end{align}
This representation is obtained from the defining representation \eqref{eq:defining_rep_su2} by realizing that the quaternions \eqref{eq:quaternions} are isomorphic to the Lie group SU(2). This follows from their commutation relations
\begin{align}
    \comm{- i \sigma_i}{- i \sigma_j} &= - 2 \sigma_{ij} = - 2 \sigma_{ijk} \sigma_k \label{eq:lie_algebra_su2}
\end{align}
which are isomorphic to those of the Lie algebra $\mathfrak{su}(2)$, since $\sigma_{ijk} = \pm i$ where the sign depends on the number of permutations needed to bring $\sigma_{ijk}$ to $i := \sigma_{123}$.\footnote{The geometric product automatically takes care of any sign swaps, hence there is no need for the Levi-Civita tensor.}
One can easily verify that the matrices defined in Eq.~\eqref{eq:matrix_rep_4d} obey the $\mathfrak{su}(2)$ Lie algebra \eqref{eq:lie_algebra_su2} as well, and as such form another representation.

Specifically, the matrices \eqref{eq:matrix_rep_4d} correspond to the $(1/2,0)\oplus (0,1/2)$-represen\-tation of the restricted Lorentz group. The representations of the underlying Lie algebra can be obtained from those of $\mathfrak{su}(2)\oplus \mathfrak{su}(2)$, see \cite{Weinberg:1995mt} for an informal physics-inspired discussion or \cite{Bargmann} for a rigorous treatment. 

In this $4 \times 4$ representation $\qK{4}$ is anti-diagonal, and therefore
    \[ \frac{1}{4} \Tr \pqty{i \qK{4}} = 0 \]
exactly, just like for $\qI{4}$ and $\qJ{4}$. This representation truly respects the orthogonality of the basis elements of $\spacealg$.

The matrix representation of a bicomplex number $w$ is now written as
    \begin{equation}
        w = (a + b i) \mathbbm{1}_n + (c + d i) \mathbf{J}_n,
    \end{equation}
where $\mathbbm{1}_n$ is the $n$-dimensional identity matrix and
    \begin{align}
        \expval{w}_\mathbb{C} := a + b i = \frac{1}{n} \tr\bqty{w}, \quad \expval{w \qJ{n}^{-1}}_\mathbb{C} := c + d i = \frac{1}{n} \tr\bqty{w \mathbf{J}_n^{-1}}.
    \end{align}
The matrix representation of the derivative of a holomorphic function $f(z)$ is given by
    \begin{align}
        f'(z) &= \frac{1}{n h}\tr[f(z \mathbbm{1}_n + h \mathbf{J}_n) \mathbf{J}_n^{-1}] + \order{h^2}.
        \label{eq:mat_f}
    \end{align}
Common mathematical functions should be replaced by their matrix counterpart. 
For example, the operation of division should be replaced by matrix inversion 
    \begin{equation}
        1 / w \rightarrow w^{-1},
    \end{equation}
where due to commutativity there is no distinction between the left and right inverse, 
and the matrix exponent is defined by the series
    \begin{equation}
        e^w = \sum_{n=0}^\infty \frac{w^n}{n!},
    \end{equation}
which can then be used to define trigonometric functions
    \begin{align}
        \sin w &= \frac{1}{2i} \pqty{e^{wi}-e^{-wi}}, \qquad \cos w = \frac{1}{2} \pqty{e^{wi}+e^{-wi}}.
    \end{align}
Such matrix functions are commonly available in programming languages.

However, not all matrix functions maintain the desired precision, as we shall see in \cref{sec:results_first}. For this reason \cref{app:elementary_functions} describes how various elementary functions like logarithms and square roots could be implemented to maintain a high degree of accuracy.

\subsection{Existence of the Taylor series} \label{sec:series}

Under which conditions does the Taylor series of \cref{eq:qsd_def} exist? A holomorphic function $f(z)$ is complex differentiable at every point $z_0$ in its domain $U$.
Therefore it can be written as a Taylor series around $z_0$:
    \begin{equation}
        f(z) = \sum_{k=0}^\infty \frac{1}{k!} \dv[k]{f(z_0)}{z} \pqty{z - z_0}^k.
    \end{equation}
This series converges when $\abs{z - z_0} < r$, with $r$ the radius of convergence.

In order to find the convergence conditions of $f(z+h \qqq)$, consider the $2 \times 2$ matrix representation of the argument:
    \begin{equation}
        A = z \mathbbm{1}_2 + h \qQ{2} .
    \end{equation}
The eigenvalues of $A$ are $\lambda_\pm = z \pm i h$.
Applying Theorem 4.7 of \cite{Higham:2008:FM} to $f(A)$, the Taylor expansion of $f(A)$ around $z_0 \mathbbm{1}_2$
exists if and only if each of the eigenvalues of $A$ satisfies one of the following conditions:
\begin{enumerate}
    \item $\abs{\lambda_\pm - z_0} = \abs{(z - z_0) \pm i h} < r$.
    \item $\abs{\lambda_\pm - z_0} = \abs{(z - z_0) \pm i h} = r$, and the series is convergent for $f(\lambda_\pm)$.
\end{enumerate}
Therefore, since $h$ is arbitrarily small,
$h$ can always be chosen such that $\abs{\lambda_\pm - z_0} \leq r$.

An alternative approach is to consider the Cauchy-Riemann equations for functions over biquaternions. 
Defining complex numbers as $z_n = a_n + i b_n$, a biquaternion can be written as $w = z_1 + \qqq z_2$. In order for the Taylor series \cref{eq:qsd_def} to be valid, we need to know under which conditions $f'(w_0) = \pdv*{f(w_0)}{w}$ exists. Let $h \in \spacealg \cong \mathbb{C} \otimes \mathbb{H}$ be a complex quaternion. Then the formal definition of $f'(w_0)$ is
    \begin{equation}
        f'(w_0) := \lim_{h \to 0} \frac{f(w_0+h) - f(w)}{h},
    \end{equation}
if this limit exists. For the limit to exist it must be the same regardless of the direction along which it is approached. This results in the Cauchy-Riemann equations
    \begin{equation}
        \pdv{f(w_0)}{a_1} = \frac{1}{i}\pdv{f(w_0)}{b_1} = \frac{1}{\qqq}\pdv{f(w_0)}{a_2} = \frac{1}{i\qqq}\pdv{f(w_0)}{b_2}. \label{eq:cr_bicomplex}
    \end{equation}
The  Wirtinger derivatives with respect to these variables are
    \begin{alignat}{3}
        \pdv{}{z_n} &= \frac{1}{2} \pqty{\pdv{}{a_n} - i \pdv{}{b_n}} \qquad &\pdv{}{\tilde{z}_n} &= \frac{1}{2} \pqty{\pdv{}{a_n} + i \pdv{}{b_n}} \\
        \pdv{}{w} &= \frac{1}{2} \pqty{\pdv{}{z_1} - \qqq \pdv{}{z_2}} \qquad &\pdv{}{\tilde{w}} &= \frac{1}{2} \pqty{\pdv{}{\tilde{z}_1} + \qqq \pdv{}{\tilde{z}_2}}.
    \end{alignat}
Using these derivatives the Cauchy Riemann equations \eqref{eq:cr_bicomplex} can be restated as independence of $f(w)$ from $\tilde{w}$, $\tilde{z}_2$ and $\tilde{z}_1$:
    \begin{equation}
        \pdv{f(w_0)}{\tilde{w}} = \pdv{f(w_0)}{\tilde{z}_1} = \pdv{f(w_0)}{\tilde{z}_2} = 0.
    \end{equation}
Therefore the Taylor expansion of \cref{eq:qsd_def} is valid when the Cauchy-Riemann equations \cref{eq:cr_bicomplex} are met.

\subsection{Relation to existing work}\label{sec:existingwork}

The use of complex analysis for the evaluation of derivatives of holomorphic functions dates back at least to Lyness~\cite{lyness} and Lyness and Moler~\cite{Moler}, based predominantly on Cauchy integrals in the complex plane. More recently, it was shown by Bornemann that optimization of the contour of Cauchy integrals may yield machine precision accuracy even for very high order derivatives \cite{bornemann}. Yet, apart from achieving similar high accuracy, the evaluation of Cauchy integrals with quadrature rules differs substantially from CSD and QSD, which require only a single function evaluation.

The Complex Step Derivative was introduced by Squire and Trapp~\cite{squire}, who with a number of examples also emphasized the possibility of achieving machine precision accuracy. The technique continues to be used (see e.g.~\cite{balzani,martins}) and extended: to higher order derivatives~\cite{lai}, matrix functions~\cite{higham}, Lie groups \cite{cossette} and bicomplex or multicomplex numbers~\cite{casado,Lantoine}.

In applications, CSD is often seen as an alternative to Automatic Differentiation (AD), with which it shares some similarities~\cite{martins2}, especially when based on dual numbers. Yet, complex numbers are available in all popular programming languages and this often makes CSD comparatively easy to implement.

Turner already generalized CSD using quaternions ~\cite{turner}. However, the three unit quaternions $\qi$, $\qj$ and $\qk$ are used separately for each argument of a three-dimensional function. Unlike the setting of holomorphic functions of a single complex argument, this does not require that they commute. Commuting roots of $-1$ are, as we have established, an essential property for the derivation of functions of a single complex variable using quaternions.

Among the above-mentioned extensions of CSD, the multicomplex methods of \cite{casado,Lantoine} are most closely related to QSD. Indeed, much like quaternions, the relatively recent theory of multicomplex analysis starts from multiple roots of $-1$ which, unlike quaternions, all commute by construction \cite{price}. These roots offer additional degrees of freedom that may be used to evaluate higher order, partial derivatives of real analytic functions. This was shown and used by Lantoine~\cite{Lantoine}, and also provides the basis for a recent Matlab toolbox~\cite{casado}. 

Having arrived independently at a similar construction, we will focus on the differences between QSD and multicomplex methods later on. We will demonstrate that QSD can be implemented using complex quaternions, geometric algebra, or a complex matrix representation thereof, and so can also be implemented relatively easily when one of these facilities is available.

\section{Results}\label{sec:results}

Four possible implementations of QSD were tested for their accuracy in computing first derivatives of holomorphic functions. To provide context, the results are compared against central differences, and the state of the art Matlab multicomplex number class of Casado and Hewson \cite{casado}.
The four QSD implementations are a geometric algebra implementation using the \verb|clifford| Python package \cite{python_clifford}, complex quaternions using \verb|SymPy| \cite{SymPy}, and the $2 \times 2$ and $4 \times 4$ matrix representations of \S\ref{sec:matrix_repr} using \verb|SciPy| \cite{2020SciPy}.

The results are presented as follows.
Firstly, in \S\ref{sec:results_first} the calculation of the first derivative of a literature standard holomorphic function is performed. 
Secondly, \S\ref{sec:results_agnostic} compares the angular dependence of the quaternionic step for the GA and matrix implementations. Thirdly, \S\ref{sec:relation_finite_diff} describes a relation to finite differences. 
Finally, \S\ref{sec:results_log} considers the differentiation of the principal logarithm $\Ln z$ since it is such a fundamental building block of complex analysis.

In all of the following sections the relative error in the $n$-th derivative $f^{(n)}$ is measured as
    \begin{equation}
        \epsilon = \frac{\abs{f^{(n)} - f^{(n)}_\text{exact}}}{\abs{f^{(n)}_\text{exact}}}.
    \end{equation}
The $n$-th order central difference is computed as
    \begin{equation}\label{eq:central_difference}
        \delta^n_h f(x) = \sum_{i=0}^n (-1)^{i} \binom{n}{i} f\pqty{x + \pqty{\frac{n}{2}-i}h },
    \end{equation}
where $h$ is the step size.
\begin{figure}
    \centering
    \begin{subfigure}[t]{.47\textwidth}
        \centering
        \includegraphics[width=\textwidth]{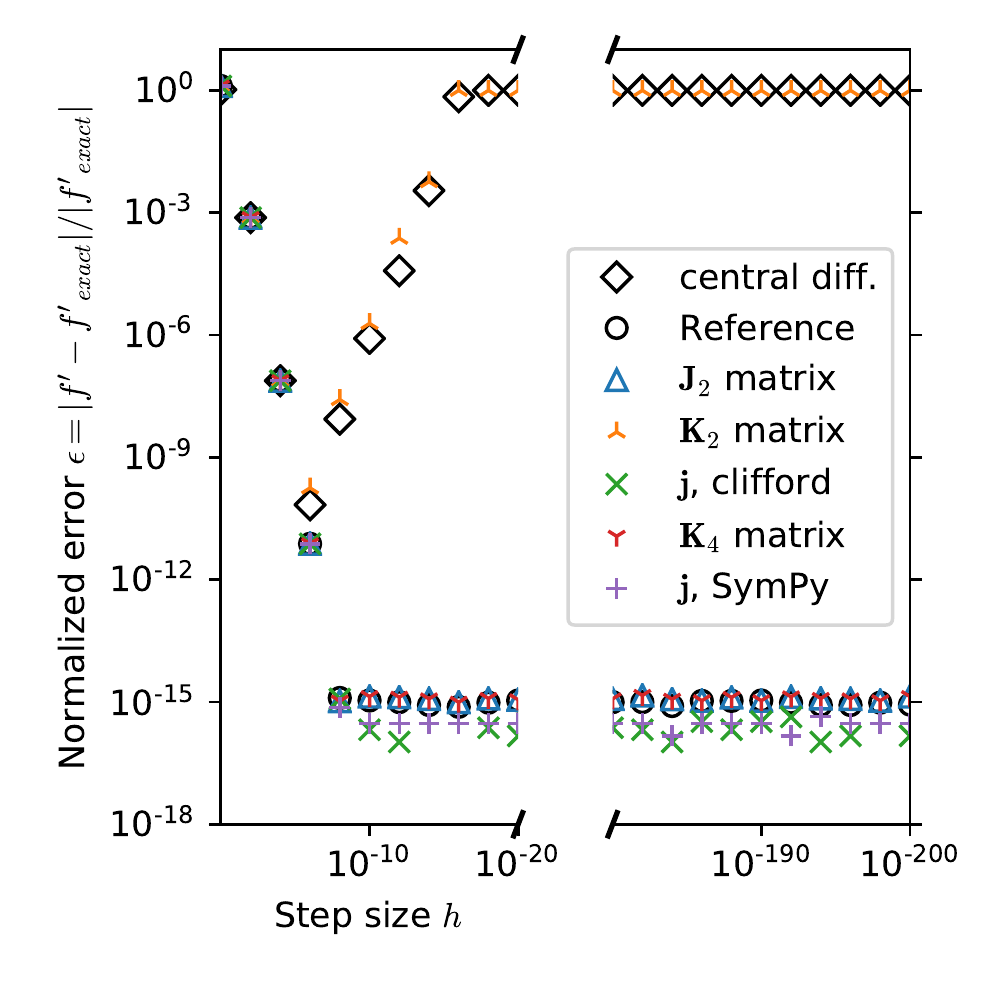}
        \caption{Precision in the calculation of $\dv{g}{z}$, evaluated at $z=\frac{\pi}{4} + \frac{\pi}{3} i$.}
        \label{fig:results_1st}
    \end{subfigure} ~ ~
    \begin{subfigure}[t]{.47\textwidth}
        \centering
        \includegraphics[width=\textwidth]{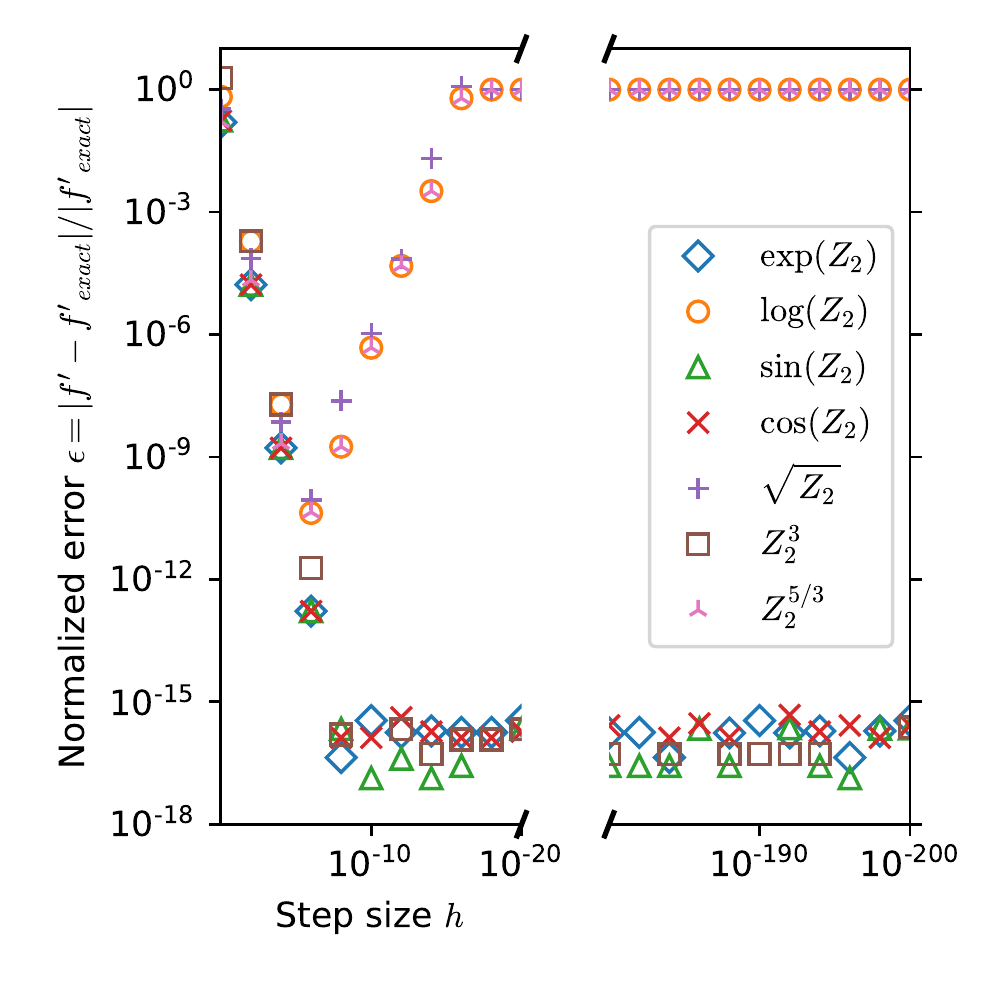}
        \caption{Derivatives of $2 \times 2$ matrix representations of various transcendental and algebraic functions, evaluated at $Z_2=(\frac{\pi}{4} + \frac{\pi}{3} i) \mathbbm{1}_2$.}
        \label{fig:results_matrix_rep}
    \end{subfigure}
    \caption{Step size $h$ vs. the relative error $\epsilon$. \Cref{fig:results_1st}: All QSD based methods leveraging (trace) orthogonality are accurate up to numerical precision and are convergent for $g(z)$. \Cref{fig:results_matrix_rep}: for matrix representations the precision of QSD becomes limited to that of the underlying matrix method.}
\end{figure}
    
\subsection{First derivative of a holomorphic function} \label{sec:results_first}
A literature standard trial function in CSD papers is
    \begin{equation}
        g(z) = \frac{e^z}{\cos^3(z) + \sin^3(z)},
        \label{eq:trial_func}
    \end{equation}
dating back to the original 1967 paper by Lyness and Moler \cite{Moler}.

Although \cref{eq:trial_func} is a holomorphic function, CSD only allows the computation of its derivative for real values, as the complex axis is reserved for holding the first derivative. But with QSD we are now able to accurately compute the derivative anywhere on the complex plane. To demonstrate this we shall differentiate $g(z)$ at
\[z = \frac{\pi}{4} + \frac{\pi}{3} i. \]

\Cref{fig:results_1st} shows the result of deriving $g$ at $z = \frac{\pi}{4} + \frac{\pi}{3} i$ with the methods listed in \S\ref{sec:results}.
The various implementations of QSD perform markedly better than the central difference method, as finite differences does not achieve the same accuracy as QSD, and does not converge as $h$ gets smaller. The QSD matrix implementation using $\qK{2} = - i \sigma_3$ performs similar to the central difference method as expected, but all three implementations of QSD which truly leverage the orthogonality of $\mathcal{G}(\mathbb{R}_3)$ result in machine precision determination of the derivative and are convergent.

When using matrix representations, the available precision in QSD becomes limited by the precision of the matrix methods used. This is shown in \cref{fig:results_matrix_rep}, which shows the $h$ vs. $\epsilon$ curve for the $2 \times 2$ matrix representations of various transcendental and algebraic functions evaluated at $z$. The \verb|SciPy| implementations of \verb|expm|, \verb|sinm| and \verb|cosm| use Padé approximation to great success, whereas the failing \verb|logm| and \verb|sqrtm| rely on Schur factorization.
This indicates that in practice, QSD does not work as desired for all matrix functions.
However, we know from the theory of complex quaternions that all of these functions are well defined. \Cref{app:elementary_functions} lists the closed form expressions for logarithms, square roots, and inverse trigonometric functions, and should be used to implement these functions to guarantee machine precision derivatives.

\subsection{Quaternionic directional dependence} \label{sec:results_agnostic}

As discussed in \S\ref{sec:agnostic}, the method is expected to function equally well for all quaternionic directions. However, in \S\ref{sec:matrix_repr} it was already discussed that QSD implemented using $\qK{2}$ can break the required orthogonality, which  was also numerically found in section \S\ref{sec:results_first}. To investigate the angular dependency of this loss of orthogonality in more detail, \cref{fig:results_angular} displays the relative error $\epsilon$ for a scan of possible quaternionic directions
    \begin{equation}
        \qqq = c_\qi \qi + c_\qj \qj + c_\qk \qk,
    \end{equation}
where $c_\qi$, $c_\qj$ and $c_\qk$ were previously defined in \cref{eq:quat_coeff}, for $h=10^{-20}$. This procedure is explained in more detail in \S\ref{sec:agnostic}.
    
The $4 \times 4$ matrix implementation shown in \cref{fig:results_angular_ga} reaches $\epsilon < 10^{-15}$ everywhere, clearly demonstrating that the quaternionic step can be taken in any direction as anticipated in \S\ref{sec:agnostic}. However, the $2\times 2$ matrix representation shown in \cref{fig:results_angular_mat} fails everywhere except when $\theta = \pi/2$, corresponding to $c_\qk = 0$ and therefore 
    \begin{equation}
        \qqq = c_\qi \qi + c_\qj \qj = \cos\phi \qi + \sin\phi \qj.
    \end{equation}
When using the $2 \times 2$ matrix representation $\qK{2}$ should therefore be avoided, but any combination of $\qI{2}$ and $\qJ{2}$ can be used.
Of course, using a $4 \times 4$ matrix representation instead of a $2 \times 2$ representation comes with a higher computational cost.

\begin{figure}
    \centering
    \begin{subfigure}{.47\textwidth}
        \centering
        \includegraphics[width=\textwidth]{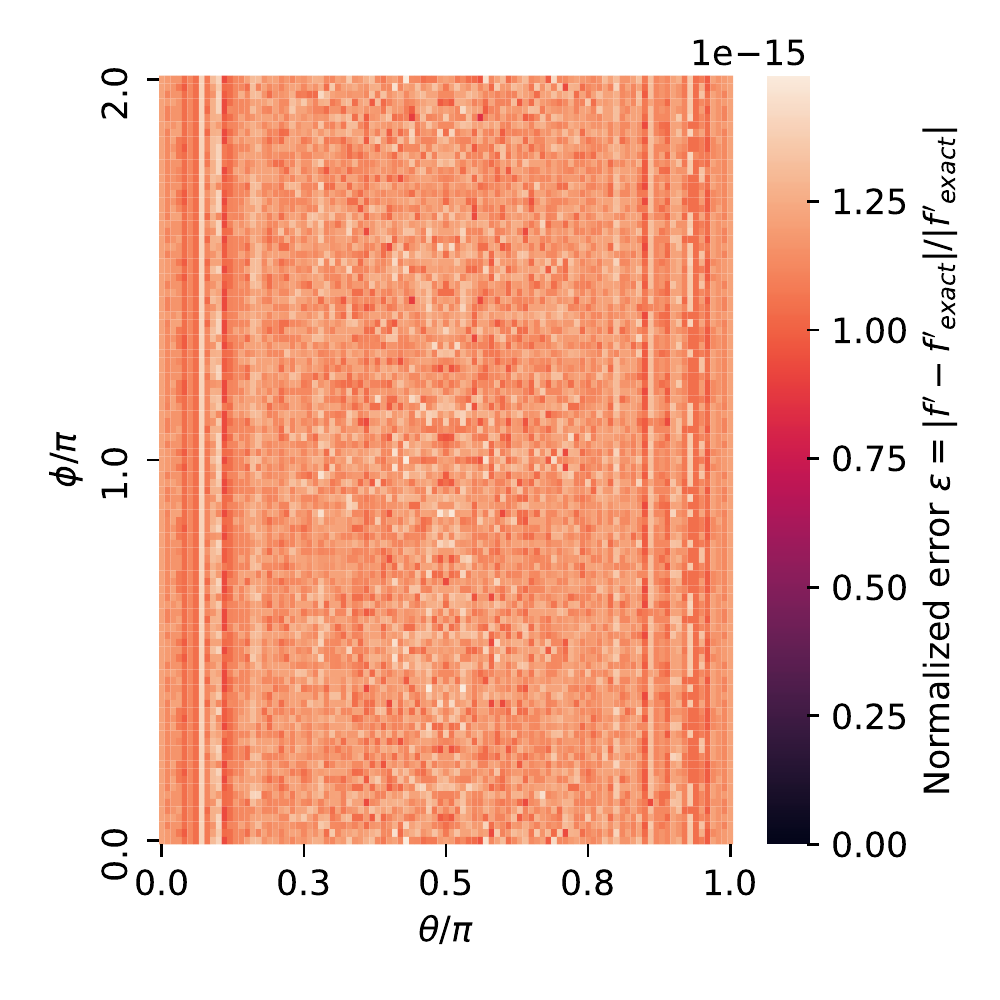}
        \caption{Angular dependence of $\dv{g}{z}$ in the $4 \times 4$ matrix implementation, evaluated at $z=\frac{\pi}{4} + \frac{\pi}{3} i$.}
        \label{fig:results_angular_ga}
    \end{subfigure} ~ ~
    \begin{subfigure}{.47\textwidth}
        \centering
        \includegraphics[width=\textwidth]{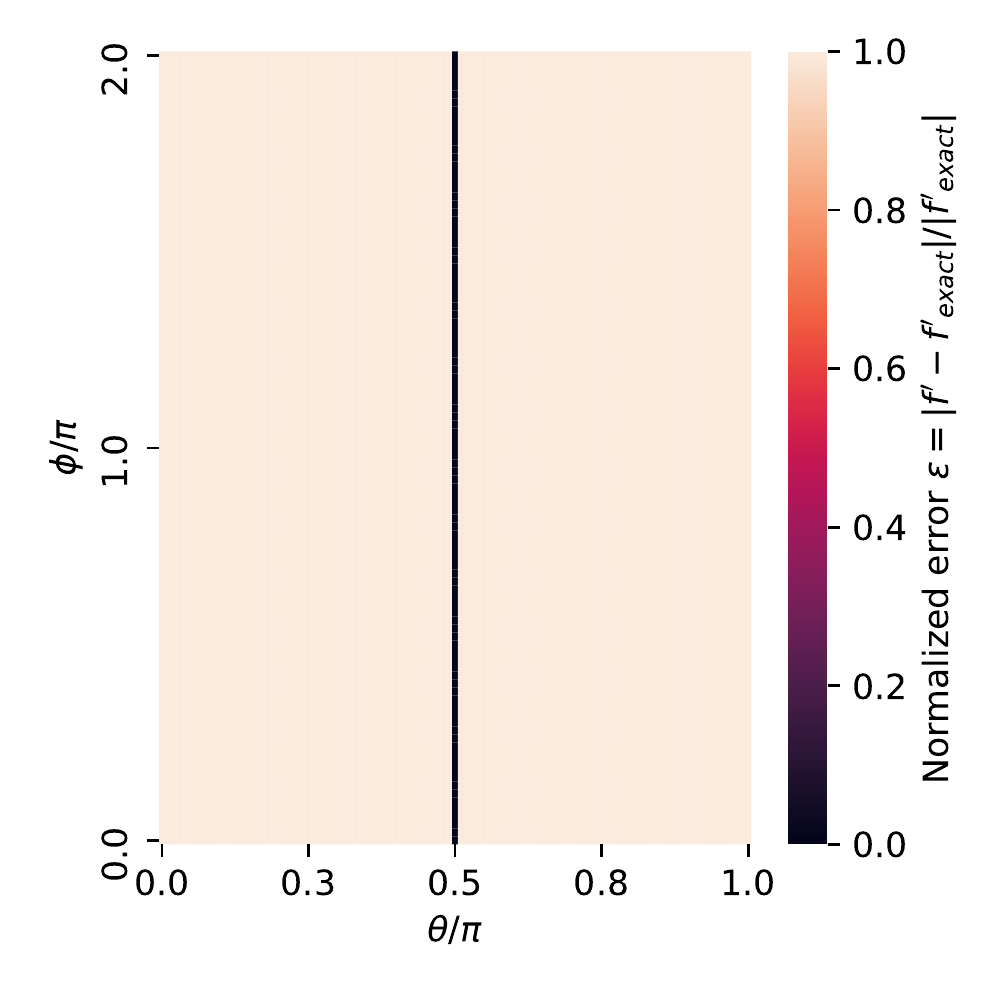}
        \caption{Angular dependence of $\dv{g}{z}$ in the $2 \times 2$ matrix implementation, evaluated at $z=\frac{\pi}{4} + \frac{\pi}{3} i$.}
        \label{fig:results_angular_mat}
    \end{subfigure}
    \caption{Angular dependence of the relative error $\epsilon$. The $4 \times 4$ matrix  implementation in \cref{fig:results_angular_ga} shows no explicit angular dependence, whereas the $2 \times 2$ matrix representation in \cref{fig:results_angular_mat} has a clear preference for $\theta = \frac{\pi}{2}$, which avoids $\qk$. Results comparable to \cref{fig:results_angular_ga} are obtained for the GA and complex quaternion implementations, and hence those results are not shown.}
    \label{fig:results_angular}
\end{figure}

\subsection{Relation between QSD and finite differences}\label{sec:relation_finite_diff}

The results in the previous sections \S\ref{sec:results_first}--\ref{sec:results_agnostic} have shown that the initial convergence of the QSD methods is comparable to that of the central difference formula \eqref{eq:central_difference}. This is not a coincidence, as the following lemma shows.

\begin{lemma}
With $\qQ{2}$ defined as in \S\ref{sec:agnostic}, \eqref{eq:mat_f} satisfies
\begin{equation}\label{eq:equivalence}
\frac{1}{2h}\tr[f(z \mathbbm{1}_2 + h \qQ{2}) \qQ{2}^{-1}] = \frac{f(z + h i) - f(z-h i)}{2h}.
\end{equation}
\end{lemma}
\begin{proof}
 Denote by $A = z \mathbbm{1}_2 + h \qQ{2}$ the matrix-argument to the function $f$ in \eqref{eq:mat_f} and note that:
 \begin{align}\label{eq:matrix_argument}
 A &= z \mathbbm{1}_2 + h \qQ{2} \\
 &= \left[
\begin{array}{cc}
 z-i h \cos (\theta ) & h \sin (\theta ) [-i \cos (\phi )-\sin (\phi )] \\
 h \sin (\theta ) [\sin (\phi )-i \cos (\phi )] & z+h i \cos (\theta ) \\
\end{array}
\right] \notag.
 \end{align}
 The eigenvalues of this matrix are $\lambda = z \mp hi$, with corresponding normalized eigenvectors 
 \[  \frac{1}{\sqrt{2} } \left[\begin{array}{c} -\frac{-\cot (\theta ) \mp \csc (\theta )}{\cos (\phi )+i \sin (\phi )} \\
 1\end{array}\right] \]
 Thus, $A = P \Lambda P^{-1}$ with
 \[
 \Lambda = \left[\begin{array}{cc}
  z - h i & 0 \\ 0 & z + h i
 \end{array}\right], \quad P = \frac{1}{\sqrt{2}}\left(
\begin{array}{cc}
 -\frac{-\cot (\theta )-\csc (\theta )}{\cos (\phi )+i \sin (\phi )} & -\frac{\csc (\theta )-\cot (\theta )}{\cos (\phi )+i \sin (\phi )} \\
 1 & 1 \\
\end{array}
\right).
 \]
 By the spectral mapping theorem, using $\qQ{2}$ again and performing the multiplications and trace operation, the result follows.
\end{proof}

The lemma shows that the application of central difference with a purely imaginary step $ih$ is mathematically equivalent to the trace of a function evaluation with a matrix argument. Yet, this mathematical equivalence falls short of explaining the beneficial convergence of QSD for very small $h$, which is a numerical phenomenon.

Indeed, the right hand side of~\eqref{eq:equivalence} suffers catastrophic cancellation for small $h$, but the left hand side does not iff $\theta = \pi /2$. In this case the matrix argument given by \eqref{eq:matrix_argument} cleanly separates quantities of different magnitude along the diagonal and anti-diagonal. As a result, this matrix can be accurately represented with standard floating point values even when $h = \num{e-200}$. Depending on how the evaluation of $f$ is implemented, the separation of orders may even persist, as for small $h$ any analytic function $f$ approximately acts separately on the diagonal and anti-diagonal -- this follows from the Taylor series \eqref{eq:taylor_quaternion} where the matrix-representation of $\mathbf{Q}$ is anti-diagonal. Conversely, if $\mathbf{Q}$ is not an anti-diagonal matrix, the quantities of different order mix and any presence of a small $h$ is numerically neglected. This is the case when using the matrix representation of $\qK{2}$, which does not work as illustrated in \S\ref{sec:results_agnostic}.

As an explicit example, with $z = 0.5 + 0.2i$, $h = \num{e-100}$, $\qQ{2} = \qJ{2}$ and $f(z)=e^z$, we see that
\[
 A = \left[\begin{array}{cc}
  0.5+0.2i & \num{e-100} \\ \num{-e-100} & 0.5+0.2i
 \end{array}\right]
\]
and
\[
 f(A) \approx \left[\begin{array}{cc}
  1.61586+0.32755i  &    \pqty{1.61586+0.32755 i} \num{e-100} \\
 \pqty{- 1.61586 - 0.32755 i} \num{e-100} &      1.61586+0.32755i
 \end{array}\right].
\]

This is much like the behaviour of the original CSD method. A complex number with an imaginary part that is tiny compared to its real part may still be represented to very high accuracy. If the function evaluation is implemented correctly for such values, the output yields the first derivative of the function to high accuracy. QSD extends this to the complex plane, with -- in the case of implementation via the $2 \times 2$ matrix representation of $\spacealg$  -- the roles of the real and imaginary part of a number replaced by the diagonal and anti-diagonal of a $2 \times 2$ matrix. In both cases, CSD and QSD, the implementation effort is limited to making sure that the single function evaluation is performed accurately.

\subsection{First derivative of the Principal Logarithm} \label{sec:results_log}

\begin{figure}
    \centering
    \begin{subfigure}{.47\textwidth}
        \centering
        \includegraphics[width=\textwidth]{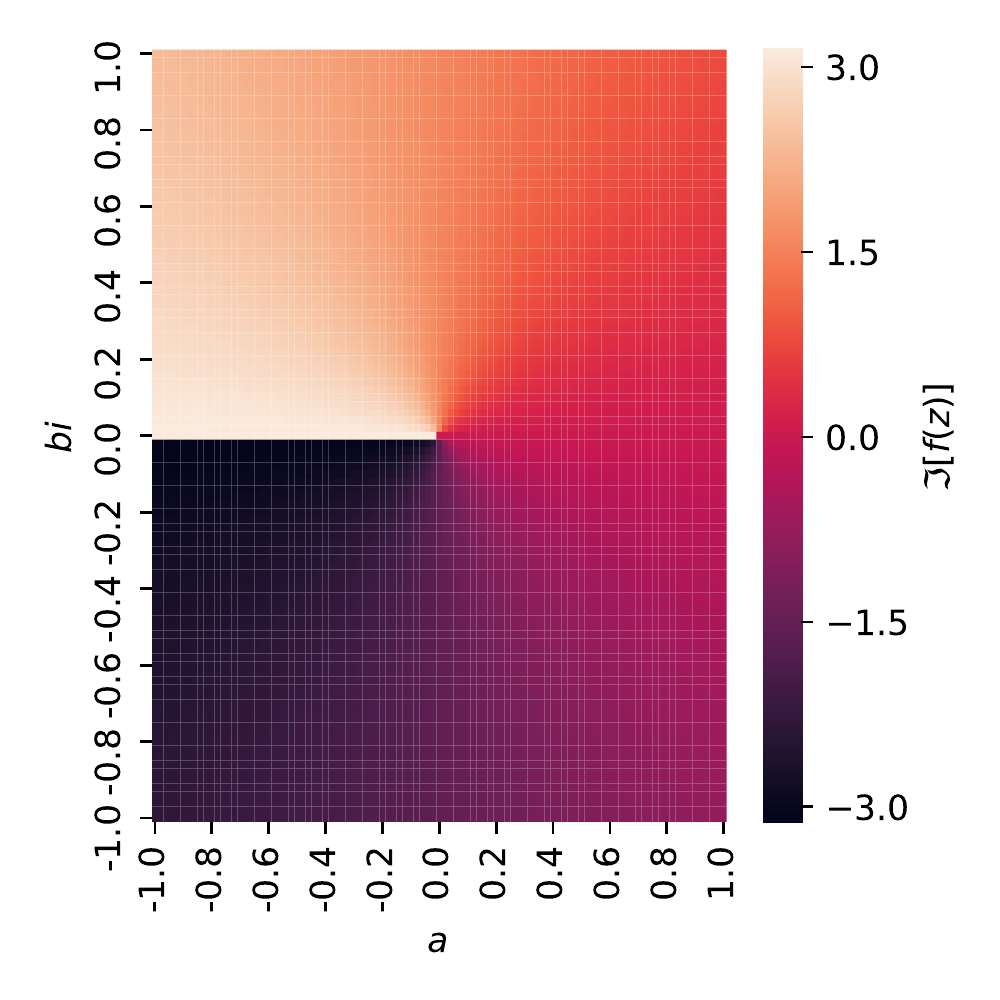}
        \caption{Imaginary part of $\Ln{z}$ on the complex plane, ranging from $-\pi$ to $\pi$. }
        \label{fig:results_im_log}
    \end{subfigure} ~ ~
    \begin{subfigure}{.47\textwidth}
        \centering
        \includegraphics[width=\textwidth]{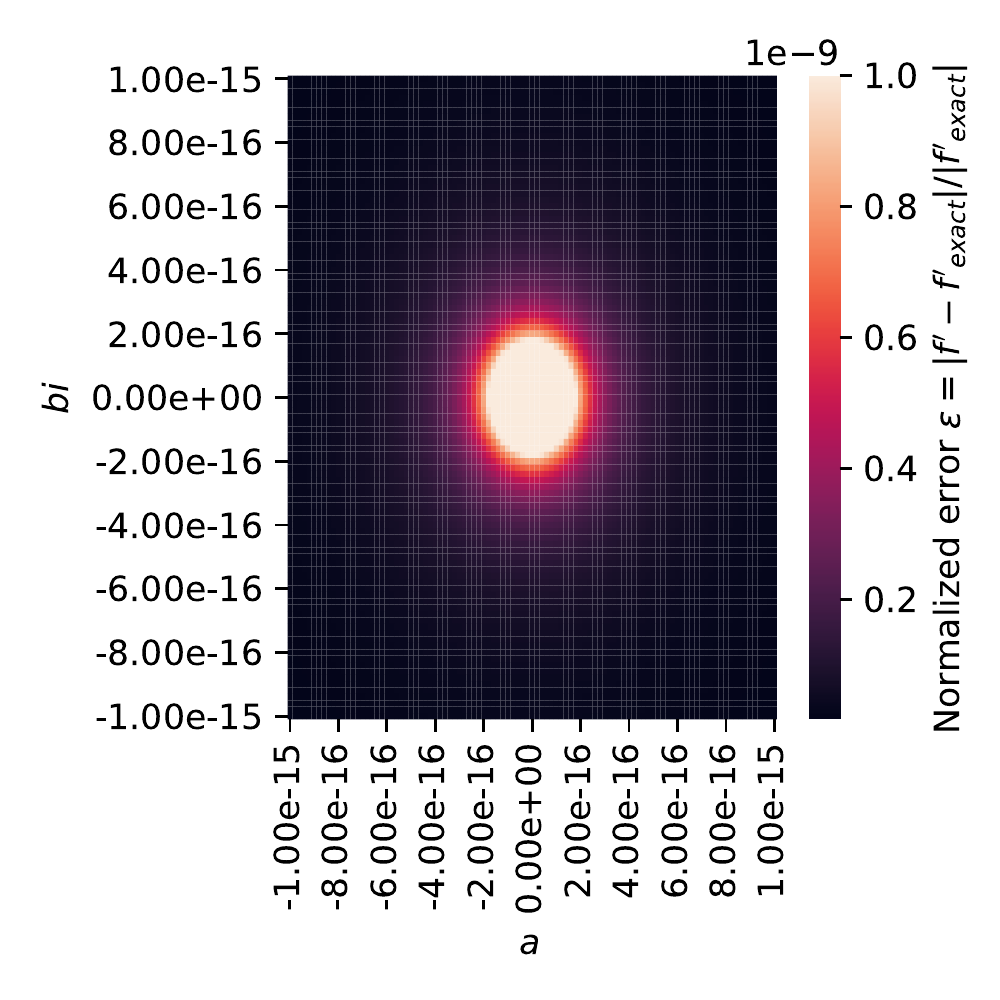}
        \caption{The derivative of $\Ln z$, focused in on the origin $z = 0$. Evaluated with $h=10^{-20}$.}
        \label{fig:results_dlog}
    \end{subfigure}
    \caption{Value of $\Ln z$ and its first derivative as obtained with QSD using $h=10^{-20}$. The singularity at $z=0$ causes numerical problems only in a very small neighborhood around the singularity.}
    \label{fig:log}
\end{figure}

A complex number $z = a + bi = r e^{i \theta}$ has principal logarithm
    \begin{equation}
        f(z) = \Ln{z} = \ln \abs{z} + i \Arg(z) = \ln r + i \theta,
    \end{equation}
where the principal argument
    \begin{equation}
        \Arg{z} = \atantwo(b, a)
    \end{equation}
lies in the open-closed interval $(-\pi, \pi]$. It is defined using the 2-argument arc\-tangent, which unambiguously returns a value in the interval $(-\pi, \pi]$. The extension of the complex logarithm when evaluated with a quaternionic component $h \qqq$ is given in \cref{app:elementary_functions}, and \cref{fig:results_matrix_rep} demonstrated that this closed form has to be used when using a matrix representation.

$\Ln z$ is analytic in the domain $r > 0$, $-\pi < \theta < \pi$, but it has a branch cut when $\theta = \pi$ and a branch point when $r=0$ \cite{ablowitz_fokas_2003}. In the domain of analyticity $\Ln z$ has derivative
    \begin{equation}
        \dv{\Ln z}{z} = \frac{1}{z}.
    \end{equation}
The derivative $1/z$ is a meromorphic function, as it is holomorphic everywhere except at the pole $z=0$.
\Cref{fig:results_im_log} shows the imaginary part of $\Ln z$ for $-1 \leq a \leq 1$ and $-1 \leq b \leq 1$, and shows the discontinuity when $a \leq 0$, $b=0$; which is due to the branch cut. A finite difference based method would fail to compute the differential across this branch cut, and for the same reason the definition of the derivative in the complex plane
    \begin{equation}
        \dv{f}{z} = \lim_{\Delta z \to 0} \frac{f(z + \Delta z) - f(z)}{\Delta z}
    \end{equation}
does not exist at the branch cut.

But \cref{fig:results_dlog} shows the relative error when applying QSD to $\Ln z$ with $h=10^{-20}$, zoomed in on the small region $-10^{-15} \leq a \leq 10^{-15}$, $-10^{-15} \leq b \leq 10^{-15}$. 
Two things stand out in the figure. 
Firstly, the singularity at $z=0$, for which the derivative $1/z$ has a pole, eventually causes the method to fail when $z \to 0$. However, even at a radius $r \approx 2 \cdot 10^{-16}$ the derivative is still computed with a relative error of only $\epsilon = 10^{-9}$, indicating excellent accuracy even near a singularity.
Secondly, the branch cut of $\Ln z$ is not visible in \cref{fig:results_dlog}. This is because the derivative is computed with only a single evaluation of $f(z + h \qqq)$, and since $-\pi < \Arg{z} \leq \pi$ includes $\pi$, the method does not experience ambiguity at the branch cut. 
This result might be unexpected considering that formally the derivative is not defined on the branch cut. 
However, this ambiguity arises precisely because the complex derivative is defined with a subtraction, which is the problem QSD was designed to avoid.
As the Riemann surface of $\Ln z$ clearly has a derivative of $1/z$ everywhere, $1/z$ is in fact the correct answer, albeit unexpected.

\section{Conclusion}\label{sec:conclusion}
Quaternionic Step Differentiation (QSD) generalizes the attractive properties of Complex Step Differentiation to holomorphic functions: it requires only a single function evaluation, the procedure is convergent, and standard mathematical operations are well defined for complex quaternions, and so any standard complex quaternion implementation facilitates QSD. To demonstrate this flexibility QSD implementations using $\spacealg$, complex quaternions and the matrix representations thereof were investigated. The method can therefore be implemented using the traditional tools of scientific computing, even in programming environments where currently no GA or AD support is available.

\section*{Acknowledgements}
The authors would like to thank Adil Han Orta, Peter C. Kroon and Steven De Keninck for valuable discussions about this research. 

\begin{appendices}
\section{Elementary Functions}\label{app:elementary_functions}

Many of the common transcendental functions take on particularly nice forms when they are considered over complex-quaternions. In this section closed form expressions for the exponential, logarithm, square root and (inverse) trigonometric functions will be given. 

\subsection{Exponents}
Of vital importance to a great number of elementary functions is the exponential function. 
For $w = a + b i + c \qqq + d \, i \qqq$ the exponential function can be greatly simplified by using $\comm{i}{\qqq} = 0$, and $i^2 = \qqq^2 = -1$, $(i \qqq)^2 = 1$:
    \begin{align}
        X &= e^{w} = e^a e^{b i} e^{c \qqq} e^{d \, i \qqq} \label{eq:exp_closed_form} \\
        &= e^a \bqty{\cos(b)+i\sin(b)}\bqty{\cos(c)+\qqq\sin(c)}\bqty{\cosh(d)+i \qqq\sinh(d)}. \notag
    \end{align}

An alternative derivation of the same result is to write $w$ more explicitly as a complex quaternion:
    \[w = (a + b i) + (c + di) \qqq\]
and using the quaternion exponentiation formula,
    \begin{align*}
        X &= e^w = e^{a+bi} e^{(c+di) \qqq} \\
        &= e^{a+bi} \bqty{\cos(c+d i) + \qqq \sin(c+di)}.
    \end{align*}

\subsection{Polar Decomposition}
Due to the commutativity of $i$ and $\qqq$, every invertible element
\[ X = e^{a + b i + (c + d i) \qqq} \] 
can trivially be decomposed into a unitary element $U = \exp(b i + c \qqq)$ and a scale factor $S = \exp(a + d\, i \qqq)$:
    \begin{align}
        X &= e^{a + b i + c \qqq + d i \qqq} = e^{a} e^{b i} e^{c \qqq}  e^{d \, i \qqq} \notag \\
        &= \pqty{ e^{b i} e^{c \qqq}} \pqty{e^{a} e^{d \, i \qqq}}  = U S,
    \end{align}
where $U \tilde{U} = 1$ and $\tilde{S} = S$ is self-reverse (Hermitian). 
The $\sim$ operation reverses the ordering of basis vectors, e.g.
    \[ \tilde{i} = e_{321} = - e_{123} = - i, \quad \tilde{\qqq} = - \qqq,\]
and corresponds to the Hermitian conjugation operation of linear algebra.
In order to find the decomposition, the first step is to realize that $X\tilde{X}$ only carries information about $S^2$:
    \begin{equation}
        X\tilde{X} = S U \tilde{U} \tilde{S} = S^2.
    \end{equation}
If the square root $S =\sqrt{S^2}$ can be found, then a decomposition into $S$ and $U$ is given by
    \[ U = X S^{-1}. \]

\begin{lemma}
        \begin{equation}
            S = \frac{X\tilde{X} + [[X\tilde{X}]]}{\sqrt{2} \sqrt{\expval{X\tilde{X}}_0 + [[X\tilde{X}]]}},
            \label{eq:polar_decomposition}
        \end{equation}
    with 
        \begin{equation}
            [[X\tilde{X}]] = \sqrt{\expval{X\tilde{X}}_0^2 - \expval{X\tilde{X}}_1^2},
            \label{eq:leonorm}
        \end{equation}
    where $\expval{M}_r$ selects the grade $r$ part of $M$.
\end{lemma}
\begin{proof}
    A detailed derivation of Eq. \eqref{eq:polar_decomposition} is given in \cite{DorstDecomposition}, but Eq. \eqref{eq:polar_decomposition} can be validated by explicit calculation:
    \begin{align}
        X\tilde{X} &= e^{2a + 2 d i \qqq} = e^{2a} \pqty{\cosh(2d) + i \qqq \sinh(2d)}
    \end{align}
    and
    \begin{align}
        [[ X\tilde{X} ]] &= e^{2a}
    \end{align}
    Plugging this into \eqref{eq:polar_decomposition} yields
        \begin{align}
            S &= \frac{e^{2a} \pqty{e^{2 d i \qqq} + 1} }{\sqrt{2} \sqrt{e^{2a} \pqty{1 + \cosh(2d)}}} \notag \\
            &= \frac{e^a}{\sqrt{2}} \pqty{\frac{1 + \cosh(2d) + i \qqq \sinh(2d)}{\sqrt{2 \cosh^2(d)}}} \notag \\
            &= \frac{e^a}{2} \pqty{\frac{2 \cosh^2(d) + 2 i \qqq \cosh(d) \sinh(d)}{\cosh(d)}} \notag \\
            &= e^a \pqty{\cosh(d) + i \qqq \sinh(d)} = e^{a + d i\qqq}
        \end{align}
\end{proof}
Since $S = \expval{S}_0 + \expval{S}_1$, i.e. it only has a scalar and a vector part, its inverse is
\begin{equation}
    S^{-1} = \frac{\expval{S}_0 - \expval{S}_1}{\expval{S}_0^2 - \expval{S}_1^2}.
\end{equation}
Therefore $U$ is obtained as
\[ U = X S^{-1}. \]

\subsection{Square Root} \label{app:sqrt}
The square root of a unitary element $U$ is found using the polar decomposition \eqref{eq:polar_decomposition} of $1+U$, for
    \begin{equation}
        X = 1 + U = \sqrt{U} \pqty{\sqrt{\tilde{U}}+\sqrt{U}},
    \end{equation}
and $S = \pqty{\sqrt{\tilde{U}}+\sqrt{U}}$ is clearly self-reverse, so
    \begin{equation}
        \sqrt{U} = X S^{-1} = (1 + U) S^{-1}.
    \end{equation}
This also makes it clear how to take the square root of any element $X$, since
    \begin{equation}
        X^{1/2} = U^{1/2} S^{1/2}.
    \end{equation}
The remaining task is to find $S^{1/2}$. Up to the usual ambiguity in sign when dealing with square roots, $S^{1/2}$ is \cite{DorstDecomposition}
    \begin{equation}
        \sqrt{S} = \frac{S + [[S]] }{\sqrt{2} \sqrt{S_0 + [[S]] }}.
        \label{eq:sqrtS}
    \end{equation}
where $[[\ldots]]$ is defined in \cref{eq:leonorm}.
As demonstrated in \cite{DorstDecomposition} this square root is unique  when $S_0 + [[S]] > 0$. But since $S_0 = e^a \cosh(d)$ and $S_1 = e^a \sinh(d )i \qj$,
    \begin{align}
        S_0 + [[S]] &= e^a \cosh(d) + \sqrt{e^{2a} \pqty{\cosh^2(d) - \sinh^2(d)}} \notag \\
        &= e^a \pqty{\cosh(d) + 1} > 0,
    \end{align}
the condition $S_0 + [[S]] > 0$ is always met in QSD, making Eq. \eqref{eq:sqrtS} unique up to a sign.

The square root of an element $X = \exp(w) = \exp\pqty{a + b i + c \qqq + d \, i \qqq}$ as derived in this section is identical to
    \begin{equation}
        X^{1/2} = e^{w / 2} = e^{\pqty{a + b i + c \qqq + d \, i \qqq}/2},
    \end{equation}
as long as $-\pi < \sqrt{b^2 + c^2} \leq \pi$ \emph{except} when $U = \pm i \qqq$, since $X = 1 \pm i\qqq$ is non-invertible. In this case the principal logarithm of \cref{app:log} must be used to find the square root, which returns $\frac{\pi}{2} (i + \qqq)$.

\subsection{Logarithm} \label{app:log}
The principal logarithm of a number $X = e^{a + b i + c \qqq + d i \qqq}$ can be found using its polar decomposition \eqref{eq:polar_decomposition}:
    \begin{equation}
        \Ln(X) = \Ln(U) + \Ln(S).
    \end{equation}
Therefore the principal logarithms of $U$ and $S$ have to be defined.
Starting with $U$, since  $U = \exp(b i)\exp(c \qqq)$, expanding using Eulers formula yields
    \begin{align}
        \expval{U}_0 &= \cosh\pqty{b i} \cosh\pqty{c \qqq} = \cos\pqty{b} \cos\pqty{c}  \notag \\
        \expval{U}_1 &= \sinh\pqty{b i} \sinh\pqty{c \qqq} = \sin\pqty{b} \sin\pqty{c} i \qqq \notag \\
        \expval{U}_2 &= \cosh\pqty{b i} \sinh\pqty{c \qqq} = \cos\pqty{b} \sin\pqty{c} \qqq \notag \\
        \expval{U}_3 &= \sinh\pqty{b i} \cosh\pqty{c \qqq} = \sin\pqty{b} \cos\pqty{c} i, \label{eq:U_grades}
    \end{align}
where $U_r = \expval{U}_r$ selects the $r$-grade part of $U$.
To take the logarithm we will have to make use of the 2-argument arctangent $\atantwo$, which returns an angle in $(-\pi, \pi]$.

The inclusion of $\qqq$ introduces some ambiguity when defining the principal logarithm. For example, both $\exp\pqty{i \pi} = -1$ and $\exp\pqty{\qqq \pi} = -1$. In order to be backwards compatible with ordinary complex analysis we demand $\Ln\pqty{-1} = i \pi$. This can be done by finding the value of $c$ using the following lemma:
\begin{lemma}
    \begin{align*}
        \sin(2c) &= 2 \pqty{U_1 U_3 - U_0 U_2} \qqq \\
        \cos(2c) &= 2 \pqty{U_0^2 - U_3^2} - 1,
    \end{align*}
    from which it follows that 
    \begin{equation}
        c = \frac{1}{2} \atantwo\pqty{\sin(2c), \cos(2c)}.
    \end{equation}
\end{lemma}
\begin{proof}
The proof follows straightforwardly from \eqref{eq:U_grades} and the trigonometric identities $\cos^2\theta + \sin^2\theta = 1$, $2 \cos\theta \sin\theta = \sin(2\theta)$ and $2 \cos^2(\theta) = 1 + \cos(2\theta)$.
\end{proof}

This limits the domain of $c$ to $(- \frac{1}{2}\pi, \frac{1}{2}\pi]$, thus removing the ambiguity surrounding the logarithm of $-1$, while at the same time still permitting $c = h \approx 0$ as it should be in QSD. With the value of $c$ at hand, $b \in (-\pi, \pi]$ is now calculated as
    \begin{equation}
        b = \atantwo\pqty{\sin(b), \cos(b)},
    \end{equation}
where
    \begin{align*}
        \sin b &= \begin{cases}
            \frac{U_3}{\cos(c) i} & \abs{\cos(c)} > \frac{1}{\sqrt{2}} \\
            \frac{U_1}{\sin(c) i \qqq} & \text{otherwise}
        \end{cases} \\
        \cos b &= \begin{cases}
            \frac{U_0}{\cos(c)} & \abs{\cos(c)} > \frac{1}{\sqrt{2}} \\
            \frac{U_2}{\sin(c) \qqq} & \text{otherwise}
        \end{cases}.
    \end{align*}
Switching the cases is technically only required when $\cos(c) = 0$, but doing so early avoids numerical issues when $\cos(c) \to 0$ and circumvents choosing an arbitrary epsilon value for when $\cos(c)$ is considered 0.

The logarithm of $S$ is more straightforwardly obtained by expanding 
\[S = e^a e^{d\, i \qqq}\] 
using Eulers formula, and solving for $a$ and $d i \qqq$:
    \begin{equation}
        a = \frac{1}{2} \ln(S_0^2 - S_1^2), \quad d i \qqq = \arctanh \pqty{\frac{S_1}{S_0}}.
    \end{equation}

Putting the principal logarithms of $U$ and $S$ together a principal logarithm for $X$ is obtained.
As a principal logarithm it is by no means unique, in fact it is $2 \pi$ periodic in both $i$ and $\qqq$.

\subsection{(Inverse) Trigonometric Functions} \label{app:trig}
Trigonometric functions of complex-quaternions
    \begin{equation}
        w = (a + b i) + \qj (c + d i) = z_1 + \qj z_2
    \end{equation}
are defined by their exponential forms:
    \begin{equation}
        \sin{w} = \frac{e^{iw} - e^{-iw}}{2 i}, \quad \cos{w} = \frac{e^{iw} + e^{-iw}}{2}, \quad \tan{w} = \frac{\sin{w}}{\cos{w}}.
    \end{equation}

Inverse trigonometric functions are defined using their logarithmic forms as familiar from complex analysis:
    \begin{align}
        \arcsin{w} &= - i \ln \pqty{i w + \sqrt{1 - w^2}}\\ 
        \arccos{w} &= - i \ln \pqty{w + \sqrt{w^2 - 1}} \\
        \arctan{w} &= \frac{i}{2} \bqty{\ln \pqty{1 - i w} - \ln \pqty{1 + i w}},
    \end{align}
where the square roots and logarithms are implemented using the techniques detailed in \cref{app:sqrt} and \cref{app:log} respectively.

\end{appendices}

\bibliographystyle{plainnat}
\bibliography{biblio}

\begin{thebibliography}{24}
\providecommand{\natexlab}[1]{#1}
\providecommand{\url}[1]{\texttt{#1}}
\expandafter\ifx\csname urlstyle\endcsname\relax
  \providecommand{\doi}[1]{doi: #1}\else
  \providecommand{\doi}{doi: \begingroup \urlstyle{rm}\Url}\fi

\bibitem[Ablowitz and Fokas(2003)]{ablowitz_fokas_2003}
M.~J. Ablowitz and A.~S. Fokas.
\newblock \emph{Complex Variables: Introduction and Applications}.
\newblock Cambridge Texts in Applied Mathematics. Cambridge University Press, 2
  edition, 2003.
\newblock \doi{10.1017/CBO9780511791246}.

\bibitem[Al-Mohy and Nigham(2010)]{higham}
A.~H. Al-Mohy and N.~J. Nigham.
\newblock The complex step approximation to the {F}r{\'e}chet derivative of a
  matrix function.
\newblock \emph{Numerical Algorithms}, 53:\penalty0 133--148, 2010.

\bibitem[Arsenovic et~al.()Arsenovic, Hadfield, Wieser, Kern, and {The Pygae
  Team}]{python_clifford}
A.~Arsenovic, H.~Hadfield, E.~Wieser, R.~Kern, and {The Pygae Team}.
\newblock pygae/clifford.
\newblock URL \url{https://doi.org/10.5281/zenodo.1453978}.

\bibitem[Balzani et~al.(2015)Balzani, Gandhi, Tanaka, and
  Schr{\"o}der]{balzani}
D.~Balzani, A.~Gandhi, M.~Tanaka, and J.~Schr{\"o}der.
\newblock Numerical calculation of thermo-mechanical problems at large strains
  based on complex step derivative approximation of tangent stiffness matrices.
\newblock \emph{Computational Mechanics}, 55:\penalty0 861--871, 2015.

\bibitem[Bargmann(1947)]{Bargmann}
V.~Bargmann.
\newblock Irreducible unitary representations of the {L}orentz group.
\newblock \emph{Annals of Mathematics}, 48\penalty0 (3):\penalty0 568--640,
  1947.

\bibitem[Bornemann(2010)]{bornemann}
F.~Bornemann.
\newblock Accuracy and stability of computing high-order derivatives of
  analytic functions by {C}auchy integrals.
\newblock \emph{Foundations of Computational Mathematics}, 11\penalty0
  (1):\penalty0 1--63, July 2010.
\newblock \doi{10.1007/s10208-010-9075-z}.

\bibitem[Casado and Hewson(2020)]{casado}
J.~M.~V. Casado and R.~Hewson.
\newblock Algorithm 1008: Multicomplex number class for matlab, with a focus on
  the accurate calculation of small imaginary terms for multicomplex step
  sensitivity calculations.
\newblock \emph{ACM Trans. Math. Softw.}, 46:\penalty0 1--26, 2020.

\bibitem[Cossette et~al.(2020)Cossette, Walsh, and Forbes]{cossette}
C.~C. Cossette, A.~Walsh, and J.~R. Forbes.
\newblock The complex-step derivative approximation on matrix {L}ie groups.
\newblock \emph{IEEE Robotics and Automation Letters}, 5\penalty0 (2):\penalty0
  906--913, 2020.

\bibitem[Dorst and Valkenburg(2011)]{DorstDecomposition}
L.~Dorst and R.~Valkenburg.
\newblock \emph{Square Root and Logarithm of Rotors in 3D Conformal Geometric
  Algebra Using Polar Decomposition}, pages 81--104.
\newblock 01 2011.
\newblock \doi{10.1007/978-0-85729-811-9_5}.

\bibitem[Dorst et~al.(2009)Dorst, Fontijne, and Mann]{GA4CS}
L.~Dorst, D.~Fontijne, and S.~Mann.
\newblock \emph{Geometric Algebra for Computer Science: An Object-Oriented
  Approach to Geometry}.
\newblock Morgan Kaufmann Publishers Inc., San Francisco, CA, USA, 2009.
\newblock ISBN 9780080553108.

\bibitem[Hestenes and Lasenby(2015)]{STA}
D.~Hestenes and A.~Lasenby.
\newblock \emph{Space-time Algebra, second edition}.
\newblock 01 2015.
\newblock \doi{10.1007/978-3-319-18413-5}.

\bibitem[Higham(2008)]{Higham:2008:FM}
N.~J. Higham.
\newblock \emph{Functions of Matrices: {Theory} and Computation}.
\newblock Society for Industrial and Applied Mathematics, Philadelphia, PA,
  USA, 2008.
\newblock ISBN 978-0-898716-46-7.

\bibitem[Lai and Crassidis(2008)]{lai}
K.~L. Lai and J.~L. Crassidis.
\newblock Extensions of the first and second complex-step derivative
  approximations.
\newblock \emph{J. Comput. Appl. Math.}, 219:\penalty0 276--293, 2008.

\bibitem[Lantoine et~al.(2012)Lantoine, Russell, and Dargent]{Lantoine}
G.~Lantoine, R.~P. Russell, and T.~Dargent.
\newblock Using multicomplex variables for automatic computation of high-order
  derivatives.
\newblock \emph{ACM Trans. Math. Softw.}, 38\penalty0 (3), April 2012.
\newblock ISSN 0098-3500.
\newblock \doi{10.1145/2168773.2168774}.
\newblock URL \url{https://doi.org/10.1145/2168773.2168774}.

\bibitem[Lyness(1967)]{lyness}
J.~N. Lyness.
\newblock Numerical algorithms based on the theory of complex variable.
\newblock In \emph{Proceedings of the 1967 22nd National Conference}, ACM '67,
  page 125–133, New York, NY, USA, 1967. Association for Computing Machinery.
\newblock ISBN 9781450374941.
\newblock \doi{10.1145/800196.805983}.
\newblock URL \url{https://doi.org/10.1145/800196.805983}.

\bibitem[Lyness and Moler(1967)]{Moler}
J.~N. Lyness and C.~B. Moler.
\newblock Numerical differentiation of analytic functions.
\newblock \emph{SIAM Journal on Numerical Analysis}, 4\penalty0 (2):\penalty0
  202--210, 1967.
\newblock \doi{10.1137/0704019}.
\newblock URL \url{https://doi.org/10.1137/0704019}.

\bibitem[Martins et~al.(2003)Martins, Sturdza, and Alonso]{martins}
J.~Martins, P.~Sturdza, and J.~Alonso.
\newblock The complex-step derivative approximation.
\newblock \emph{ACM Trans. Math. Softw.}, 29:\penalty0 245--262, 2003.

\bibitem[Martins et~al.(2001)Martins, Sturdza, and Alonso]{martins2}
J.~R. R.~A. Martins, P.~Sturdza, and J.~J. Alonso.
\newblock The connection between the complex-step derivative approximation and
  algorithmic differentiation.
\newblock In \emph{Proceedings of the 39th AIAA Aerospace Sciences Meeting and
  Exhibit}. 2001.
\newblock Paper AIAA-2001-0921.

\bibitem[Meurer et~al.(2017)Meurer, Smith, Paprocki, \v{C}ert\'{i}k, Kirpichev,
  Rocklin, Kumar, Ivanov, Moore, Singh, Rathnayake, Vig, Granger, Muller,
  Bonazzi, Gupta, Vats, Johansson, Pedregosa, Curry, Terrel, Rou\v{c}ka, Saboo,
  Fernando, Kulal, Cimrman, and Scopatz]{SymPy}
A.~Meurer, C.~P. Smith, M.~Paprocki, O.~\v{C}ert\'{i}k, S.~B. Kirpichev,
  M.~Rocklin, A.~Kumar, S.~Ivanov, J.~K. Moore, S.~Singh, T.~Rathnayake,
  S.~Vig, B.~E. Granger, R.~P. Muller, F.~Bonazzi, H.~Gupta, S.~Vats,
  F.~Johansson, F.~Pedregosa, M.~J. Curry, A.~R. Terrel, \v{S}. Rou\v{c}ka,
  A.~Saboo, I.~Fernando, S.~Kulal, R.~Cimrman, and A.~Scopatz.
\newblock Sympy: symbolic computing in python.
\newblock \emph{PeerJ Computer Science}, 3:\penalty0 e103, January 2017.
\newblock ISSN 2376-5992.
\newblock \doi{10.7717/peerj-cs.103}.
\newblock URL \url{https://doi.org/10.7717/peerj-cs.103}.

\bibitem[Price(1991)]{price}
G.~B. Price.
\newblock \emph{An introduction to multicomplex spaces and functions}.
\newblock Marcel Dekker Inc., New York, 1991.

\bibitem[Squire and Trapp(1998)]{squire}
W.~Squire and G.~Trapp.
\newblock Using complex variables to estimate derivatives of real functions.
\newblock \emph{SIAM Review}, 40:\penalty0 110--112, 1998.

\bibitem[Turner(2002)]{turner}
J.~D. Turner.
\newblock Quaternion-based partial derivative and state transition matrix
  calculations for design optimization.
\newblock In \emph{Proceedings of the 40th AIAA Aerospace Sciences Meeting and
  Exhibit}. 2002.
\newblock Paper AIAA-2002-0448.

\bibitem[{Virtanen} et~al.(2020){Virtanen}, {Gommers}, {Oliphant}, {Haberland},
  {Reddy}, {Cournapeau}, {Burovski}, {Peterson}, {Weckesser}, {Bright}, {van
  der Walt}, {Brett}, {Wilson}, {Jarrod Millman}, {Mayorov}, {Nelson}, {Jones},
  {Kern}, {Larson}, {Carey}, {Polat}, {Feng}, {Moore}, {VanderPlas}, {Laxalde},
  {Perktold}, {Cimrman}, {Henriksen}, {Quintero}, {Harris}, {Archibald},
  {Ribeiro}, {Pedregosa}, {van Mulbregt}, and {Contributors}]{2020SciPy}
P.~{Virtanen}, R.~{Gommers}, T.~E. {Oliphant}, M.~{Haberland}, T.~{Reddy},
  D.~{Cournapeau}, E.~{Burovski}, P.~{Peterson}, W.~{Weckesser}, J.~{Bright},
  S.~J. {van der Walt}, M.~{Brett}, J.~{Wilson}, K.~{Jarrod Millman},
  N.~{Mayorov}, A.~R.~J. {Nelson}, E.~{Jones}, R.~{Kern}, E.~{Larson},
  CJ~{Carey}, {\.I}.~{Polat}, Y.~{Feng}, E.~W. {Moore}, J.~{VanderPlas},
  D.~{Laxalde}, J.~{Perktold}, R.~{Cimrman}, I.~{Henriksen}, E.~A. {Quintero},
  C.~R. {Harris}, A.~M. {Archibald}, A.~H. {Ribeiro}, F.~{Pedregosa}, P.~{van
  Mulbregt}, and SciPy 1.~0 {Contributors}.
\newblock {SciPy 1.0: Fundamental Algorithms for Scientific Computing in
  Python}.
\newblock \emph{Nature Methods}, 17:\penalty0 261--272, 2020.
\newblock \doi{https://doi.org/10.1038/s41592-019-0686-2}.

\bibitem[Weinberg(2005)]{Weinberg:1995mt}
S.~Weinberg.
\newblock \emph{{The Quantum theory of fields. Vol. 1: Foundations}}.
\newblock Cambridge University Press, Cambridge, 2005.
\newblock ISBN 978-0-521-67053-1, 978-0-511-25204-4.

\end{thebibliography}

\end{document}